\documentclass[12pt]{amsart}
\usepackage{amsfonts,amsthm,latexsym,amsmath,amssymb,amscd,amsmath, mathrsfs, url, cite, filecontents, mathtools}
\newcounter{count}

\numberwithin{count}{section}
\newtheorem{Lemma}[count]{Lemma}

\newtheorem{Corollary}[count]{Corollary}
\newtheorem{Theorem}[count]{Theorem}

\begin{document}

\author[T. H.~Nguyen]{Thu Hien Nguyen}

\address{Department of Mathematics \& Computer Sciences, V. N. Karazin Kharkiv National University,
4 Svobody Sq., Kharkiv, 61022, Ukraine}
\email{nguyen.hisha@karazin.ua}

\author[A.~Vishnyakova]{Anna Vishnyakova}
\address{Department of Mathematics \& Computer Sciences, V. N. Karazin Kharkiv National University,
4 Svobody Sq., Kharkiv, 61022, Ukraine}
\email{anna.vishnyakova@karazin.ua}

\title[The entire functions of the Laguerre-P\'olya I class]
{On the entire functions from the Laguerre-P\'olya I class with non-monotonic second quotients of Taylor coefficients}

\begin{abstract}
We study the entire functions $f(z) = \sum_{k=0}^\infty a_k z^k, a_k>0,$  
with non-monotonic second quotients of Taylor coefficients, namely, such that
$\frac{a_{2m-1}^2}{a_{2m-2}a_{2m}} = a>1$ and 
$\frac{a_{2m}^2}{a_{2m-1}a_{2m+1}} = b>1$ for all $m \in \mathbb{N}.$
We obtain necessary and sufficient conditions under which  such  functions  
belong to the Laguerre-P\'olya I class.

\end{abstract}



\maketitle

\section{Introduction}
To begin with, we provide the definitions of the Laguerre--P\'olya class and Laguerre--P\'olya class of type I.

{\bf Definition 1}. {\it A real entire function $f$ is said to be in the {\it
Laguerre--P\'olya class}, written $f \in \mathcal{L-P}$, if it can
be expressed in the form
\begin{equation}
\label{lpc}
 f(z) = c z^n e^{-\alpha z^2+\beta z}\prod_{k=1}^\infty
\left(1-\frac {z}{x_k} \right)e^{zx_k^{-1}},
\end{equation}
where $c, \alpha, \beta, x_k \in  \mathbb{R}$, $x_k\ne 0$,  $\alpha \ge 0$,
$n$ is a nonnegative integer and $\sum_{k=1}^\infty x_k^{-2} < \infty$. }

{\bf Definition 2}. {\it  A real entire function $f$ is said to be in the {\it Laguerre--
P\'olya class of type I}, 
written $f \in \mathcal{L-P} I$, if it can
be expressed in the following form  
\begin{equation}  \label{lpc1}
 f(z) = c z^n e^{\beta z}\prod_{k=1}^\infty
\left(1+\frac {z}{x_k} \right),
\end{equation}
where $c \in  \mathbb{R},  \beta \geq 0, x_k >0 $, 
$n$ is a nonnegative integer,  and $\sum_{k=1}^\infty x_k^{-1} <
\infty$.   }

Note that the product on the right-hand sides in both definitions can be
finite or empty (in the latter case, the product equals 1).

These classes are essential in the theory of entire functions since it appears that the polynomials with only real zeros (or only real and nonpositive zeros) converge locally 
uniformly to these and only these functions. The following  prominent theorem provides an even stronger result. 

{\bf Theorem A} (E.~Laguerre and G.~P\'{o}lya, see, for example,
\cite[p. ~42--46]{HW}) and \cite[chapter VIII, \S 3]{lev}). {\it   

(i) Let $(P_n)_{n=1}^{\infty},\  P_n(0)=1, $ be a sequence
of real polynomials having only real zeros which  converges uniformly on the disc 
$|z|\leq A, A > 0.$ Then this sequence converges locally uniformly in $\mathbb{C}$ 
to an entire function 
from the $\mathcal{L-P}$ class.

(ii) For any $f \in \mathcal{L-P}$ there exists a sequence of real polynomials 
with only real zeros, which converges locally uniformly to $f$.

(iii) Let $(P_n)_{n=1}^{\infty},\  P_n(0)=1, $ be a sequence
of real polynomials having only real negative zeros which  
converges uniformly on the disc $|z| \leq A, A > 0.$ Then this 
sequence converges locally uniformly in $\mathbb{C}$ to an entire function
 from the class $\mathcal{L-P}I.$
 
(iv) For any $f \in \mathcal{L-P}I$ there is a
sequence of real polynomials with only real nonpositive 
zeros, which converges locally uniformly to $f$.}

Further, we define the second quotients of Taylor coefficients of $f.$
Let  $f(z) = \sum_{k=0}^\infty a_k z^k$  be an entire function with 
real nonzero coefficients, then 

\begin{align*}
q_n=q_n(f)&:=  \frac{a_{n-1}^2}{a_{n-2}a_n}, \quad n\geq 2.
\end{align*}

From this definition it follows straightforwardly that
\begin{align*}
& a_n = a_1\Big(\frac{a_1}{a_0} \Big)^{n-1} \frac{1}{q_2^{n-1}q_3^{n-2}
\cdot \ldots \cdot q_{n-1}^2 q_n}, \quad n\geq 2.
\end{align*}

In general, the problem of understanding whether a given entire function has only 
real zeros is not trivial. However, in 1926, J. I. Hutchinson found the following simple 
sufficient condition in terms of coefficients for an entire function with positive coefficients to 
have only real zeros.

{\bf Theorem B} (J. ~I. ~Hutchinson, \cite{hut}). { \it Let $f(x)=
\sum_{k=0}^\infty a_k x^k$, $a_k > 0$ for all $k$. 
Then $q_n(f)\geq 4$, for all $n\geq 2,$  
if and only if the following two conditions are fulfilled:\\
(i) The zeros of $f(x)$ are all real, simple and negative, and \\
(ii) the zeros of any polynomial $\sum_{k=m}^n a_k x^k$, $m < n,$  formed 
by taking any number 
of consecutive terms of $f(x) $, are all real and non-positive.}

For some extensions of Hutchinson's results see,
for example, \cite[\S4]{cc1}.

Next, we define the multiplier sequence.

{\bf Definition 3}.  A sequence $(\gamma_k)_{k=0}^\infty$ of real
numbers is called a multiplier sequence if, whenever the real
polynomial $P(x) = \sum_{k=0}^n a_k z^k $ has only real zeros, the
polynomial $\sum_{k=0}^n \gamma_k a_k z^k $ has only real zeros. 
The class of multiplier sequences is denoted by $\mathcal{MS}$.

The following theorem fully describes multiplier sequences.

{\bf Theorem C} (G. P\'olya and J.Schur, cf. \cite{polsch}, \cite[pp. 100-124]{pol}
and\cite[pp. 29-47]{O} ).  
{\it Let  $(\gamma_k)_{k=0}^\infty$ be a given 
real sequence. The following 
three statements are equivalent.}

{\it

1. $(\gamma_k)_{k=0}^\infty$ is a multiplier sequence.

2.  For every $n\in \mathbb{N}$ the
polynomial $P_n(z) =\sum_{k=0}^n {\binom{n}{k}} \gamma_k z^k $ has only real 
zeros of the same sign. 

3. The power series $ \Phi (z) := \sum_{k=0}^\infty \frac
{\gamma_k}{k!}z^k$ converges absolutely in the whole complex plane
and the entire function $\Phi(z)$ or the entire function
$\Phi(-z)$ admits the representation
\begin{equation}
\label{pred}
 C e^{\sigma z} z^m \prod_{k=1}^\infty
(1+\frac{z}{x_k}),
\end{equation}

where $C\in{\mathbb{R}}, \sigma \geq 0, m\in {\mathbb{N}}\cup \{0\}, 0<x_k
\leq \infty,\   \sum_{k=1}^\infty \frac{1}{x_k} < \infty.$}

Strikingly, the following fact is an obvious consequence.

{\bf Corollary of Theorem C.}   {\it The sequence 
$(\gamma_0, \gamma_1, \ldots, \gamma_l,$ $ 0, 0, \ldots )$ is a
multiplier sequence if and only if the polynomial $P(z)=
\sum_{k=0}^l \frac {\gamma_k}{k!}z^k$ has only real zeros of the
same sign.}

Further, let us introduce the notion of a complex zero decreasing sequence. For a real polynomial $P $  we denote by $Z_{\mathbb{C}}(P)$ the number
of nonreal zeros of $P$ counting multiplicities.

{\bf Definition 4}.  A sequence $(\gamma_k)_{k=0}^\infty$ of real
numbers is said to be a complex zero decreasing sequence (we write 
$(\gamma_k)_{k=0}^\infty \in \mathcal{CZDS}$), if
\begin{equation}
\label{czds}
 Z_{\mathbb{C}}\left(\sum_{k=0}^n \gamma_k a_k z^k\right) 
 \leq Z_{\mathbb{C}}\left(\sum_{k=0}^n a_k z^k\right) ,
\end{equation}
for any real polynomial $\sum_{k=0}^n a_k z^k .$ 

The existence of nontrivial 
$\mathcal{CZDS}$ sequences is a consequence of the following remarkable 
theorem proved by Laguerre and extended by P\'olya.

{\bf Theorem D}  ( E.~Laguerre, see
\cite[pp. 314-321]{pol}).   {\it Let $f$ be an entire function from the
 Laguerre-P\'olya class  having only negative zeros. Then
$\left(f(k)\right)_{k=0}^\infty \in \mathcal{CZDS}$.}

As it follows from the  theorem above,
\begin{equation}
\label{m1}
\left(a^{-k^2}\right)_{k=0}^\infty\in\mathcal{CZDS},\  a\geq 1,\quad
\left(\frac{1}{k!}\right)_{k=0}^\infty\in \mathcal{CZDS}.
\end{equation}

A special entire function $g_a(z) =\sum _{k=0}^{\infty} z^k a^{-k^2}$, $a>1,$
known as the \textit{partial theta function} (the classical Jacobi theta function is defined 
by the series $\theta(z) := \sum_{k = - \infty}^{\infty} z^k a^{-k^2}$),  was investigated 
by many mathematicians. Note that  $q_n(g_a)=a^2$ for all $n.$ 
The  survey \cite{War} by S.O.~Warnaar contains the history of
investigation of the partial theta-function and some of its main properties. 

Note that, since $\left(a^{-k^2}\right)_{k=0}^\infty\in\mathcal{CZDS}\ $  for  $a\geq 1,$
we conclude that for every $n\geq 2$ there exists a constant $c_n >1$ such that  
$S_{n}(z,g_a):=\sum _{k=0}^{n} z^k a^{- k^2} \in \mathcal{L-P}$
$ \ \Leftrightarrow \ a^2 \geq c_n$. The following theorem answers the question 
for which values of a parameter $a$ the partial theta-function and its Taylor 
sections belong to the $\mathcal{L-P}$ class.

{\bf Theorem E} (O. ~Katkova, T. ~Lobova, A. ~Vishnyakova, \cite{klv}).  
{\it There exists a constant 
$q_\infty $ $(q_\infty\approx 3{.}23363666 \ldots ) $ such that:
\begin{enumerate}
\item
$g_a(z) \in \mathcal{L-P} \Leftrightarrow \ a^2\geq q_\infty ;$
\item
$g_a(z) \in \mathcal{L-P} \Leftrightarrow \ $  there exists $z_0 \in (- a^3, -a)$ 
such that $ \  g_a(z_0) \leq 0$
\item
for a given $n\geq 2$ we have $S_{n}(z,g_a) \in \mathcal{L-P}$ $ \  \Leftrightarrow \ $
there exists $z_n \in (- a^3, -a)$ such that $ \ S_{n}(z_n,g_a) \leq 0;$
\item
$ 4 = c_2 > c_4 > c_6 > \cdots $  and     $\lim_{n\to\infty} c_{2n} = q_\infty ;$
\item
$ 3= c_3 < c_5 < c_7 < \cdots $  and     $\lim_{n\to\infty} c_{2n+1} = q_\infty .$
\end{enumerate}}

We would like to mention a series of works by V.P. ~Kostov dedicated to the interesting properties of zeros of 
the partial theta-function and its  derivative (see \cite{kos0, kos1, kos2, kos3, kos03, 
kos04, kos4, kos5, kos5.0}). Besides, a wonderful paper \cite{kosshap} among the other results explains the role 
of the constant $q_\infty $ in the study of the set of entire functions with positive coefficients having all Taylor truncations with only real zeros. See also \cite{sokal}
for the  properties of the leading root of the partial theta function.

Subsequently, we need the following Lemma from the work \cite{ngthv2}.

{\bf Lemma F} (see \cite[Lemma 2.1]{ngthv2} or \cite[Lemma 1.2]{ngthv3}).
If $f(z) = \sum_{k=0}^\infty a_kz^k, a_k>0,$ belongs to $\mathcal{L-P}I,$ 
then $q_3(q_2 - 4) + 3 \geq 0.$
In particular, if $q_3 \geq q_2,$ then $q_2 \geq 3.$

It appears that for many important entire functions with positive
coefficients $f(z)=\sum_{k=0}^\infty a_k z^k $ (for example, the partial theta function
from \cite{klv}, functions from \cite{BohVish} and \cite{Boh}, 
the $q$-Kummer function $\prescript{}{1}{\mathbf{\phi}}_1(q;-q; q,-z)$ 
and others) the following two conditions
are equivalent: 

(i) $f$ belongs to the Laguerre--P\'olya class of type I,  

and 

(ii) there exists $x_0 \in [-\frac{a_1}{a_2},0]$ such that $f(x_0) \leq 0.$

The following theorem is  a necessary condition for an entire function to belong to the Laguerre--Pol\'ya class of type I, in terms of the closest to zero roots. We will further use it in our proofs.

{\bf Theorem G}  (T.~H.~Nguyen, A.~Vishnyakova, \cite{ngthv3}).
{\it Let $f(z)=\sum_{k=0}^\infty a_k z^k $, $a_k > 0$ for all $k,$  be an 
entire function.  Suppose that the quotients $q_n(f)$ satisfy the following condition: 
$q_2(f) \leq q_3(f).$ If the function $f$ belongs to the  Laguerre--P\'olya class of type I, then 
there exists  $z_0 \in [-\frac{a_1}{a_2},0]$ such that $f(z_0) \leq 0$. }

Many important entire functions with monotonic second quotients of Taylor 
coefficients were previously studied. For instance, 
$f(z) = \sum_{k=0}^{\infty} \frac{z^k}{(k!)a^{k^2}}, a\geq 1,$ has $q_k(f) = \frac{k}{k-1} a^2 $ which 
are decreasing in $k.$ It is known that $f  \in \mathcal{L-P I}$ for all $a \geq 1.$ In 
\cite{klv1} it is proved that all Taylor sections of this function belong to the  
Laguerre--P\'olya class of type I if and only if $a^2\geq q_\infty. $
The following function is known as the second $q$-exponential  function $E_q(z)$ with $q=1/a:$
$$h_a(z) = 1 + \sum \limits_{k=1}^\infty \frac{z^k}{(a^k -1)(a^{k-1} - 1)\cdots (a-1)} = 
\prod \limits_{k=1}^\infty \left(1 + \frac{z}{a^k} \right),  a > 1.
$$
We have $q_k(h_a) = \frac{a^k -1}{a^{k-1}-1},$ and $q_k(h_a)$  are decreasing in $k.$ Note that this 
function has only real negative zeros, i.e. it belongs to the Laguerre--P\'olya I class.
Next, the function $\varphi_{m, a} (z) = \sum_{k=0}^{\infty} \frac{z^k}{a^{k^2}} (k!)^m, a > 1, 
m \geq 1 $ was studied by A.~Bohdanov and A.~Vishnyakova in \cite{BohVish}. 
Note that  $q_k(\varphi_{m, a})$ are increasing in $k.$ The necessary and sufficient conditions 
for this function to belong to the Laguerre--P\'olya class of type I were investigated.
The function $$y_a(z) = \sum_{k=0}^{\infty} \frac{z^k}{(a + 1)(a^2 + 1) \cdots (a^k +1)}, a > 1,$$ 
is also known as the $q$-Kummer function $\prescript{}{1}{\mathbf{\phi}}_1(q;-q; q,-z).$ 
Note that its  $q_k$ are increasing in $k.$ The question of its belonging to the Laguerre--P\'olya 
class of type I was investigated by T.H.~Nguyen in \cite{ngth1}.

In \cite{ngthv1}, we have proved that if an entire function with positive coefficients
has a decreasing sequence of second quotients of Taylor coefficients and
the limit of this sequence is greater than or equal to $q_\infty,$ then this 
function belongs to the Laguerre--P\'olya 
class of type I. In \cite{ngthv2}, \cite{ngthv4} and \cite{ngthv5}, we have found the conditions for entire 
functions  with increasing second quotients of Taylor coefficients to belong to the 
Laguerre--P\'olya class of type I.

In this paper, we study the entire functions with non-monotonic second 
quotients and find the conditions for them to belong to the Laguerre-P\'olya class of type I.
We present our first attempt at investigating entire functions with non-monotonic 
second quotients and consider the following function:

$$f(x) = \sum_{k=0}^\infty \frac{x^k}{q_2^{k-1}q_3^{k-2} \cdots q_{k-1}^2q_k},$$
where 
$$q_2 = q_4 = q_6 = \ldots = a > 1,$$
$$q_3 = q_5 = q_7 = \ldots = b > 1,$$
or
\begin{align*}
f(x) = 1 + x + \frac{x^2}{a} + \frac{x^3}{a^2b} + \frac{x^4}{a^3b^2a} + 
\frac{x^5}{a^4b^3a^2b} + \ldots
\end{align*}

We look into the case when $a \neq b.$ Note that when $a = b,$ we obtain the case 
of the partial-theta function. We find the new conditions for which $(a, b)$ the entire function 
$f$ belongs to the Laguerre-P\'olya I class.

By Lemma F, if $f\in \mathcal{L-P I},$ $q_2(f) = a, q_3(f) = b$ and $a < b,$ then $a \geq 3.$

Therefore, we look for $(a, b)$ such that $a < b$ and $a \geq 3.$ The case $4 \leq a < b$ 
was studied by Hutchinson (see Theorem B), so we find new conditions for $a \in [3, 4).$

We present our main result.

\begin{Theorem}
\label{th:mthm1}
Let $f(z)=\sum_{k=0}^\infty a_k z^k $, $a_k > 0$ for all $k,$  be an 
entire function. Suppose that the quotients $q_n(f)$ satisfy the following 
condition: $q_2(f) = q_4(f) = q_6(f) = \ldots = a,$ and 
$q_3(f) = q_5(f) = q_7(f) = \ldots = b,$ $1 < a < b$. Then the function 
$f$ belongs to the Laguerre--P\'olya I  class  if and only if there exists 
$z_0 \in [-\frac{a_1}{a_2},0]$ such that $f(z_0) \leq 0$.
\end{Theorem}

The following theorem is a sufficient 
condition for the existence of such a point $z_0 $  for 
the case $q_2(f) < 4.$

{\bf Theorem H}  (T.~H.~Nguyen, A.~Vishnyakova,  \cite{ngthv3}).
{\it Let $f(z)=\sum_{k=0}^\infty a_k z^k $, $a_k > 0$ for all $k,$  be an 
entire function and  $3  \leq  q_2(f) < 4, q_3(f) \geq 2,$ and $q_4(f) \geq 3.$ If $q_3(f) 
\leq \frac{8}{d(4-d)},$ where 
$d = \min(q_2(f), q_4(f)),$ then  there exists $z_0 \in [-\frac{a_1}{a_2},0]$ 
such that $f(z_0) \leq 0$.}  

Hence, in our case, the sufficient condition is  $b \leq \frac{8}{a(4-a)}.$

The following theorem  gives  a necessary condition for an entire function to belong to the 
Laguerre--Pol\'ya class of type I, in terms of the second quotients of its Taylor coefficients $q_n$.

{\bf Theorem I}  (T.~H.~Nguyen, A.~Vishnyakova, \cite{ngthv3}).
{\it If $f(z) = \sum_{k=0}^\infty  a_k z^k,$ $a_k > 0$ for all $k,$   
belongs to the Laguerre--P\'olya I class, $q_2(f) < 4$ and $q_2(f) \leq q_3(f),$  then 
$$q_3(f) \leq \frac{- q_2(f)(2q_2(f)-9)+2(q_2(f)-3)\sqrt{q_2(f)(q_2(f)-3)}}{q_2(f)(4-q_2(f))}.$$ }

Thus, in our case, the necessary condition is 
$b \leq \frac{- a(2a-9)+2(a-3)\sqrt{a(a-3)}}{a(4-a)}.$

Our next result is the Theorem below.

\begin{Theorem}
\label{th:mthm2}
Let $f_{a,b}(z)=\sum_{k=0}^\infty a_k z^k $, $a_k > 0$ for all $k,$  be an 
entire function such that $q_2(f) = q_4(f) = q_6(f) = \ldots = a,$ and 
$q_3(f) = q_5(f) = q_7(f) = \ldots = b,$ $1 < a < b$. Then the following
statements are valid.

1. If $f$ belongs 
to the Laguerre--P\'olya I class, then $a \geq q_{\infty}.$ 

2. If the numbers $a, b, q_{\infty} \leq a <b,$ are such that 
$f_{a,b} \in \mathcal{L-P I},$ then for every $c, a < c< b$
we have $f_{a,c} \in \mathcal{L-P I}.$

3. If the numbers $a, b, q_{\infty} \leq a <b,$ are such that 
$f_{a,b} \in \mathcal{L-P},$ then for every $d, a < d< b$
we have $f_{d,b} \in \mathcal{L-P I}.$

\end{Theorem}

Note that, by Theorem C and Theorem D, every new entire function from
the $\mathcal{L-P } I$ class generates a new multiplier sequence and
a new complex zero decreasing sequence. So, we obtain the following 
direct corollary from Theorem \ref{th:mthm1}.

\begin{Corollary}
\label{th:mthm3}
Let $f(z)=\sum_{k=0}^\infty a_k z^k $, $a_k > 0$ for all $k,$  be an 
entire function. Suppose that the quotients $q_n(f)$ satisfy the 
following condition: $q_2(f) = q_4(f) = q_6(f) = \ldots = a,$ and 
$q_3(f) = q_5(f) = q_7(f) = \ldots = b,$ $1 < a < b$. Suppose that 
there exists $z_0 \in [-\frac{a_1}{a_2},0]$ such that $f(z_0) \leq 0$ (for example,
this condition is fulfilled if   $a <b \leq \frac{8}{a(4-a)}$). Then
$(k! a_k)_{k=0}^\infty \in \mathcal{MS}$  and 
$(f(k))_{k=0}^\infty \in \mathcal{CZDS}.$
\end{Corollary}

\section{Proof of Theorem \ref{th:mthm1}}

For an entire function $f(z) = \sum_{k=0}^\infty a_k z^k$  with 
real positive coefficients, without loss of generality, we can assume that $a_0=a_1=1,$ since we can 
consider a function $g(x) =a_0^{-1} f (a_0 a_1^{-1}x) $  instead of 
$f(x),$ due to the fact that such rescaling of $f$ preserves its property of 
having real zeros and preserves the second quotients:  $q_n(g) =q_n(f)$ 
for all $n.$ Throughout the paper, we use notation $p_n$ and $q_n$ instead of 
$p_n(f)$ and $q_n(f).$  

Thereafter, we consider a function  $$\varphi(x) = f(-x) = 1 - x + \sum_{k=2}^\infty  \frac{ (-1)^k x^k}
{q_2^{k-1} q_3^{k-2} \cdots q_{k-1}^2 q_k}$$  instead of $f.$

Firstly, in our case, $q_2 = a, q_3 = b, a < b,$  according to Theorem~F,  if 
$\varphi \in \mathcal{L-P}I,$ then there exists such a point $z_0 \in [0,  -\frac{a_1}{a_2}] 
= [0, a] $ that $\varphi (z_0) \leq 0.$

To prove the converse statement, we need the following lemma.

\begin{Lemma}
\label{th:lm1}
Let $\varphi(x) = \sum_{k=0}^\infty (-1)^k a_k x^k,$
$a_k >0,  k=0, 1, 2, \ldots ,$ be an entire function 
such that $$q_2 = q_4 = q_6 = \ldots = a \geq 3,$$
$$q_3 = q_5 = q_7 = \ldots = b > a.$$
For an arbitrary integer $j \geq 2$ we define 
$$\rho_j := q_2 q_3 \cdots  q_j \sqrt{q_{j+1}}. $$ 
Then, for all even $j = 2s, s\in \mathbb{N}$,  the function $\varphi$ has 
exactly $j$ zeros on the disk $\{x\  :\   |x| <  \rho_j \}$
counting multiplicities.
\end{Lemma}

\begin{proof}
For $j \geq 2$ we have

$$
\varphi(x) = \sum_{k=0}^\infty \frac{(-1)^k x^k}{q_2^{k-1}q_3^{k-2} \cdots q_k} = 
\bigg(\sum_{k=0}^{j-3} + \sum_{k = j-2}^{j+2} + 
\sum_{k = j+3}^{\infty} \bigg)=: $$
$$ \Sigma_{1,j}(x) + g_j(x) + \Sigma_{2,j}(x).
$$

In more details,
\begin{align}
\label{ll1}
&  g_j(x)  =   \frac{(-1)^{j-2} x^{j-2}}{q_2^{j-3}q_3^{j-4} \cdots q_{j-2}} 
+ \frac{(-1)^{j-1} x^{j-1}}{q_2^{j-2}q_3^{j-3} \cdots q_{j-2}^2q_{j-1}} 
+ \\  \nonumber & \frac{(-1)^{j} x^{j}}{q_2^{j-1}q_3^{j-2} \cdots q_{j-2}^3q_{j-1}^2q_{j}} 
+ \frac{(-1)^{j +1} x^{j + 1}}{q_2^{j}q_3^{j-1} \cdots q_{j-2}^4 q_{j-1}^3 q_{j}^2 q_{j+1} } +
\\  \nonumber  & \frac{(-1)^{j +2} x^{j + 2}}{q_2^{j+1}q_3^{j} 
\cdots q_{j-2}^5 q_{j-1}^4 q_{j}^3 q_{j+1}^2 q_{j+2} }  =  
\frac{(-1)^{j-2} x^{j-2}}{q_2^{j-3}q_3^{j-4} \cdots q_{j-2}} \cdot \left( 1 - \right.
 \frac{x}{q_2 q_3 \cdots q_{j-2} q_{j-1} } \\  \nonumber  & +  
 \frac{x^2}{q_2^2 q_3^2 \cdots q_{j-2}^2 q_{j-1}^2 q_j }
 - \frac{x^3}{q_2^3 q_3^3 \cdots q_{j-2}^3 q_{j-1}^3 q_j^2 q_{j+1} } +
 \frac{x^4}{q_2^4 q_3^4 \cdots q_{j-2}^4 q_{j-1}^4 q_j^3 q_{j+1}^2 q_{j+2} } \left. \right). 
\end{align}

By the definition of $\rho_j$ we have $q_2q_3 \cdots q_j  < 
\rho_j < q_2q_3 \cdots q_j q_{j+1}$.
We get
\begin{align}
\label{ll2}
&
g_j(\rho_je^{i\theta}) = 
(-1)^{j-2}e^{i(j-2)\theta}q_2q_3^2 \cdots q_{j-2}^{j-3}q_{j-1}^{j-2}q_j^{j-2}
q_{j+1}^{\frac{j-2}{2}}  \\ \nonumber &\quad\times
 \bigg( 1 - q_j \sqrt{q_{j+1}}e^{i\theta}
+ q_jq_{j+1}e^{2i\theta} - q_j\sqrt{q_{j+1}}e^{3i\theta} + 
\frac{q_j}{q_{j+2}}e^{4i\theta}\bigg) \\  \nonumber &
= (-1)^{j-2}e^{i(j-2)\theta}q_2q_3^2 \cdots 
q_{j-2}^{j-3}q_{j-1}^{j-2}q_j^{j-2}q_{j+1}^{\frac{j-2}{2}}  \\
\nonumber & \quad \times \left( 1 - e^{i\theta}q_j \sqrt{q_{j+1}}
+ e^{2i\theta}q_jq_{j+1} - e^{3i\theta}q_j\sqrt{q_{j+1}} + 
e^{4i\theta}\right) ,
\end{align}
since, by our assumptions, $q_j = q_{j+2}.$

Our aim is to show that for every even $j$ the following inequality holds:  
\begin{align*}
\min_{0\leq \theta\leq 2\pi}|g_j(\rho_j e^{i\theta})| > 
\max_{0\leq \theta\leq 2\pi}| \varphi
(\rho_j e^{i\theta}) -g_j(\rho_j e^{i\theta})|, 
\end{align*} 
so that the number of zeros of  $\varphi$  in the circle $\{x : |x| < \rho_j  \}$  
is equal to  the number of zeros of $g_j$ in the same circle. 
Subsequently in the proof, we also find the number of zeros of $g_j$ in this circle. 
First, we find $\min_{0 \leq \theta \leq 2\pi}  |g_j(\rho_je^{i\theta})|.$
We obtain
\begin{align}
\label{ll3}
&
g_j(\rho_je^{i\theta}) 
=  (-1)^{j-2}e^{ij\theta}q_2q_3^2 \cdots 
q_{j-2}^{j-3}q_{j-1}^{j-2}q_j^{j-2}q_{j+1}^{\frac{j-2}{2}} \\
 \nonumber &\quad \times  \bigg( 2\cos 2\theta - 
2\cos\theta q_j \sqrt{q_{j+1}}
+ q_jq_{j+1} \bigg) \\
 \nonumber & =: (-1)^{j-2} e^{ij\theta}q_2q_3^2 \cdots q_{j-2}^{j-3}
q_{j-1}^{j-2}q_j^{j-2}q_{j+1}^{\frac{j-2}{2}} 
\cdot  \psi_j(\theta).
\end{align}

We consider $\psi_j(\theta)$ as following:
\begin{align*}
\psi_j(\theta) =\widetilde{\psi}_j(t) := 4t^2 - 2q_j
\sqrt{q_{j+1}}t + (q_jq_{j+1} - 2),
\end{align*}
where $t:= \cos \theta,$ and where we have used 
that $\cos 2\theta = 2t^2 - 1.$

The vertex of the parabola is $t_j = q_j\sqrt{q_{j+1}}/4.$ Under our assumptions, 
$q_2 = q_4 = q_6 = \ldots = a \geq 3,$ and
$q_3 = q_5 = q_7 = \ldots = b > a,$  so that  $t_j > 1.$
Hence, $\min_{t \in [-1, 1]} \widetilde{\psi}_j(t) = \widetilde{\psi}_j(1) =
2 - 2q_j\sqrt{q_{j+1}} + q_j q_{j+1} = q_j\sqrt{q_{j+1}}(\sqrt{q_{j+1}} -2)  + 2.$
We want to show that $\widetilde{\psi}_j(1) >0,$ whence $\min_{t \in [-1, 1]} |\widetilde{\psi}_j(t) |
= \widetilde{\psi}_j(1).$ 
If  $q_{j+1} \geq 4,$ then $q_j\sqrt{q_{j+1}}(\sqrt{q_{j+1}} -2)  + 2>0. $
If $q_{j+1} <4$, taking into account the fact that $j$ is an even integer, we have
 $q_j = a, q_{j+1} = b, a < b <4,$ and then
$$ q_j\sqrt{q_{j+1}}(\sqrt{q_{j+1}} -2)  + 2 = a \sqrt{b} (\sqrt{b} - 2) + 2$$
$$ \geq a \sqrt{a} (\sqrt{a} - 2) + 2 = a^2 - 2a \sqrt{a} + 2.
$$

 Next, denote by $y = \sqrt{a} \geq 0,$ and
$g(y) = y^4 - 2 y^3 +2.$ It is easy to calculate that $\min_{y \geq 0} g(y) =
g(\frac{3}{2}) = \frac{5}{16} >0.$
Consequently,  we get
$$2 - 2q_j\sqrt{q_{j+1}} + q_j q_{j+1}> 0.$$
Thus, $\widetilde{\psi}_j(t) > 0$ for all $t \in [-1, 1]$, and we have obtained the estimate from below:
\begin{eqnarray}
\label{estg}
& \min_{0\leq \theta\leq 2\pi}|g_j(\rho_j e^{i\theta})| = q_2q_3^2 \cdots 
q_{j-2}^{j-3}q_{j-1}^{j-2}q_j^{j-2}q_{j+1}^{\frac{j-2}{2}} 
\\ \nonumber & \times  \bigg(2 - 2q_j\sqrt{q_{j+1}} 
+ q_jq_{j+1} \bigg).
\end{eqnarray}

Second, we estimate the modulus of $\Sigma_1$ from above. We have
\begin{align}
\label{ll4}
&
|\Sigma_1(\rho_j e^{i \theta})| 
\leq  \sum_{k = 0}^{j-3} \frac{q_2^kq_3^k \cdots
q_j^k q_{j+1}^{\frac{k}{2}}}{q_2^{k-1}q_3^{k-2} \cdots q_k}= 
\\   \nonumber  &\quad(\mbox{we rewrite 
the sum from right} 
 \mbox{  to left})\\   \nonumber &\quad
 = \left( q_2 q_3^2\cdots q_{j-3}^{j-4}q_{j-2}^{j-3}q_{j-1}^{j-3} q_j^{j-3} 
q_{j+1}^{\frac{j-3}{2}} + \right.
\left. q_2 q_3^2 \cdots q_{j-4}^{j-5} q_{j-3}^{j-4}q_{j-2}^{j-4} 
q_{j-1}^{j-4} q_j^{j-4} q_{j+1}^{\frac{j-4}{2}}  \right. \\
 \nonumber &\qquad + \left.   q_2 q_3^2 \cdots q_{j-5}^{j-6} q_{j-4}^{j-5}
q_{j-3}^{j-5}q_{j-2}^{j-5}q_{j-1}^{j-5} q_j^{j-5} 
q_{j+1}^{\frac{j-5}{2}} +\cdots    \right) \\
 \nonumber &\quad = q_2 q_3^2 \cdots q_{j-3}^{j-4}q_{j-2}^{j-3}q_{j-1}^{j-3} 
 q_j^{j-3} q_{j+1}^{\frac{j-3}{2}} \\
 \nonumber &\qquad \times \left( 1+ \frac{1}{q_{j-2}q_{j-1}q_{j}
 \sqrt{q_{j+1}}}\right. \left.    +  \frac{1}{q_{j-3}q_{j-2}^2 
 q_{j-1}^2q_{j}^2(\sqrt{q_{j+1}})^2} +
 \cdots\right) \\
 \nonumber & \quad \leq   q_2 q_3^2 \cdots q_{j-3}^{j-4}q_{j-2}^{j-3}
 q_{j-1}^{j-3} q_j^{j-3}q_{j+1}^{\frac{j-3}{2}}\cdot 
 \frac{1}{1- \frac{1}{q_{j-2}q_{j-1}q_{j}\sqrt{q_{j+1}}}}
\end{align}

 (we estimate the finite sum  from above by the sum of the infinite geometric 
 progression). Finally, we obtain
\begin{align}
& |\Sigma_1(\rho_j e^{i \theta})| 
\leq q_2 q_3^2 \cdots q_{j-3}^{j-4}q_{j-2}^{j-3} 
q_{j-1}^{j-3} q_j^{j-3} q_{j+1}^{\frac{j-3}{2}}\times 
 \frac{1}{1- \frac{1}{q_{j-2}q_{j-1}q_{j}\sqrt{q_{j+1}}}} .
\end{align}

Next, the estimation of $|\Sigma_2(\rho_j e^{i\theta})|$ from above can be made analogously:
\begin{multline*}
|\Sigma_2(\rho_j e^{i \theta})|
\leq  \sum_{k = j+3}^\infty \frac{q_2^kq_3^k \cdots 
q_j^k q_{j+1}^{\frac{k}{2}}}{q_2^{k-1}q_3^{k-2} \cdots q_k} 
= \frac{q_2q_3^2 \cdots q_j^{j-1} q_{j+1}^{\frac{j-3}{2}}}{q_{j+2}^2q_{j+3}}\\
\times\bigg( 1 + \frac{1}{\sqrt{q_{j+1}}q_{j+2}q_{j+3}q_{j+4}} + \frac{1}{(\sqrt{q_{j+1}})^2 
q_{j+2}^2q_{j+3}^2q_{j+4}^2q_{j+5}} + \cdots \bigg).
\end{multline*}
The latter can be estimated from above by the sum of the geometric progression, so, we obtain
\begin{equation}
|\Sigma_2(\rho_j e^{i \theta})|\leq \frac{q_2q_3^2 \cdots q_j^{j-1} 
q_{j+1}^{\frac{j-3}{2}}}{q_{j+2}^2 q_{j+3}}
\times \frac{1}{1 - \frac{1}{\sqrt{q_{j+1}}q_{j+2} q_{j+3} q_{j+4}}}.
\end{equation}

Therefore, the desired inequality $\min_{0\leq \theta\leq 2\pi}|g_j(\rho_j e^{i\theta})| > 
\max_{0\leq \theta\leq 2\pi}|\varphi (\rho_j e^{i\theta}) -g_j(\rho_j e^{i\theta})|$ 
follows from
\begin{align*}
& q_2q_3^2 \cdots q_{j-2}^{j-3}q_{j-1}^{j-2}q_j^{j-2}q_{j+1}^{\frac{j-2}{2}} 
\cdot  \bigg(2 - 2q_j\sqrt{q_{j+1}} + q_jq_{j+1} \bigg) \\ 
&\quad > q_2 q_3^2 \cdots q_{j-3}^{j-4}q_{j-2}^{j-3} 
q_{j-1}^{j-3} q_j^{j-3} q_{j+1}^{\frac{j-3}{2}}
\times \frac{1}{1- \frac{1}{q_{j-2}q_{j-1}q_{j}\sqrt{q_{j+1}}}} \\
&\qquad + \frac{q_2q_3^2 \cdots q_j^{j-1} 
q_{j+1}^{\frac{j-3}{2}}}{q_{j+2}^2 q_{j+3}}
\times \frac{1}{1 - \frac{1}{\sqrt{q_{j+1}}q_{j+2} q_{j+3} q_{j+4}}} .
\end{align*} 
Or, equivalently, 
\begin{align}
\label{estqq}
& q_{j-1}q_j\sqrt{q_{j+1}} \bigg(2 - 2q_j\sqrt{q_{j+1}} + q_jq_{j+1}\bigg)\\  
\nonumber & \quad > \frac{1}{1 - \frac{1}{q_{j-2}q_{j-1}q_j\sqrt{q_{j+1}}}} + 
\frac{q_{j-1}q_j^2}{q_{j+2}^2q_{j+3}} \cdot \frac{1}{1 - 
\frac{1}{\sqrt{q_{j+1}}q_{j+2}q_{j+3}q_{j+4}}} = \\  
\nonumber &   \frac{1}{1 - \frac{1}{q_{j-2}q_{j-1}q_j\sqrt{q_{j+1}}}} + 
 \frac{1}{1 - 
\frac{1}{\sqrt{q_{j+1}}q_{j+2}q_{j+3}q_{j+4}}}
\end{align}
by our assumptions on the sequence $(q_k)_{k=2}^\infty.$ Since $j$ is an even number,
we have $q_{j-2}=q_j = q_{j+2} =q_{j+4}=a$ and $q_{j-1}= q_{j+1} =q_{j+3}= b,$
so the last inequality has the form

\begin{equation}
\label{esta} 
b \sqrt{b} a (2 - 2 \sqrt{b} a + ab) > 
\frac{2}{1 - \frac{1}{b \sqrt{b} a^2}}. 
\end{equation}

Since $a \geq 3, b \geq 3$, the desired inequality (\ref{esta}) follows from
\begin{align*}
9\sqrt{3} (2 - 2 \sqrt{b} a + ab) > \frac{2}{1 - \frac{1}{27 \sqrt{3}}}
\end{align*}

or, equivalently, 
\begin{align*}
ab  - 2 \sqrt{b} a + 2 - \frac{6}{27\sqrt{3} - 1} > 0.
\end{align*}

After refactoring, we have
\begin{align*}
a (\sqrt{b} - 1)^2 - a  + 2 - \frac{6}{27\sqrt{3} - 1} > 0.
\end{align*}

Next, we divide both sides of the inequality by $a$ and get

\begin{align*}
(\sqrt{b} - 1)^2 > 1 + \frac{8 - 54\sqrt{3}}{a(27\sqrt{3} - 1)}.
\end{align*}

Since $8 - 54\sqrt{3}<0, $  the inequality above follows from
\begin{align*}
(\sqrt{\alpha} - 1)^2 > 1 + \frac{8 - 54\sqrt{3}}{4 \cdot (27\sqrt{3} - 1)}.
\end{align*}
Numerical calculations shows that $1 + \frac{8 - 54\sqrt{3}}{4 \cdot (27\sqrt{3} - 1)} < 0.53279.$
Thus, the inequality is fulfilled for $\sqrt{\alpha}  > 1{.}72993$, and, therefore, for 
$\alpha > 2{.}99266.$ Under our assumptions, $\alpha \geq 3 ,$ so the inequality (\ref{esta})
is valid.

Consequently, we have proved that for all even $j$
$\min_{0\leq \theta\leq 2\pi}|g_j(\rho_j e^{i\theta})| > 
\max_{0\leq \theta\leq 2\pi} |\varphi (\rho_j e^{i\theta}) -g_j(\rho_j e^{i\theta})|, $ 
so the number of zeros of 
$\varphi$ in the circle $\{x:|x| < \rho_j\}$ is equal to the number of zeros of 
$g_j$ in this circle.

In the next stage of the proof, it remains to find the number of zeros 
of $g_j$ in the circle $\{x:|x| < \rho_j\}$. 

Let us use the denotation $w = x \rho_j^{-1},$ so that $|w|<1.$ This yields

\begin{multline*}
g_j(\rho_j w) = (-1)^{j-2}w^{j-2}q_2q_3^2 \cdots 
q_{j-2}^{j-3}q_{j-1}^{j-2}q_j^{j-2}q_{j+1}^{\frac{j-2}{2}} \\
\times(1 - q_j\sqrt{q_{j+1}}w + q_jq_{j+1}w^2 - q_j\sqrt{q_{j+1}}w^3 + w^4).
\end{multline*}

Therefore, it follows from (\ref{estg}) that $g_j$ does not have zeros on the circumference $\{x : |x| = \rho_j  \},$
whence $g_j(\rho_j w)$ does not have zeros on the circumference $\{w : |w| = 1  \}.$ Since 
$P_j(w) =1-q_j\sqrt{q_{j+1}} w +q_j q_{j+1} w^2 - q_j\sqrt{q_{j+1}} w^3 +w^4 $  is a 
self-reciprocal polynomial in $w,$ we can conclude that $P_j$ has exactly two zeros 
in the circle $\{ w :|w| <1 \}.$
Hence, $g_j(x)$ has exactly $j$ zeros in the circle $\{x : |x| < \rho_j  \}$  for all even $j$, and we have 
proved the statement of  Lemma~\ref{th:lm1}.
\end{proof}

\begin{Lemma}
\label{th:lm3} Let $\varphi(x) = \sum_{k=0}^\infty (-1)^k a_k x^k,$
$a_k >0,  k=0, 1, 2, \ldots ,$ be an entire function 
such that $$q_2 = q_4 = q_6 = \ldots = a \geq 3,$$
$$q_3 = q_5 = q_7 = \ldots = b > a.$$
For an arbitrary integer $j \geq 2$ we define 
$$\rho_j := q_2 q_3 \cdots  q_j \sqrt{q_{j+1}}. $$ 
Then for every even  $j =2s, s\in \mathbb{N},$  
the following inequality holds: 
$$ \varphi(\rho_j) \geq 0.$$
\end{Lemma}

\begin{proof}
Since  $ \rho_j  \in (q_2q_3 \cdots q_j, q_2q_3 \cdots q_j q_{j+1}),$ we have 
\begin{align*}
 1< \rho_j < \frac{\rho_j^2}{q_2} < \cdots < \frac{\rho_j^j}{q_2^{j-1}q_3^{j-2} \cdots q_j},  
 \end{align*}
and
\begin{align*}
 \frac{\rho_j^j}{q_2^{j-1}q_3^{j-2} \cdots q_j} > 
 \frac{\rho_j^{j+1}}{q_2^{j}q_3^{j-1} \cdots q_j^2 q_{j+1}} 
> \frac{\rho_j^{j+2}}{q_2^{j+1}q_3^{j} 
\cdots q_j^3 q_{j+1}^2 q_{j+2}} > \cdots.
\end{align*}
Therefore, we get for even $j\geq 2$
\begin{align*}
\varphi(\rho_j) \geq \sum_{k=j-3}^{j+3}\frac{(-1)^k
\rho_j^k}{q_2^{k-1} q_3^{k-2} \cdots q_k} =: \mu_j(\rho_j),
\end{align*}
and it is sufficient to prove that for every even $j\geq 2$ we have $\mu_j(\rho_j)\geq 0.$  
After factoring out $\rho_j^{j-3}/q_2^{j-4} q_3^{j-5} \cdots q_{j-3}$ the 
desired inequality is expressed in the following form
\begin{align*}
&-1 + \frac{\rho_j}{q_2q_3 \cdots q_{j-3}q_{j-2}} - \frac{\rho_j^2}{q_2^2q_3^2 
\cdots q_{j-2}^2q_{j-1}} + \frac{\rho_j^3}{q_2^3q_3^3 \cdots q_{j-2}^3q_{j-1}^2q_j} \\
&\quad -\frac{\rho_j^4}{q_2^4q_3^4 \cdots q_{j-2}^4q_{j-1}^3q_j^2q_{j+1}} + 
\frac{\rho_j^5}{q_2^5q_3^5 \cdots q_{j-2}^5q_{j-1}^4q_j^3q_{j+1}^2q_{j+2}} \\
&\quad -\frac{\rho_j^6}{q_2^6q_3^6 \cdots q_{k-2}^6q_{j-1}^5q_j^4q_{j+1}^3q_{j+2}^2q_{j+3}} 
\geq 0,
\end{align*}
or, using that $\rho_j= q_2q_3 \cdots q_j \sqrt{q_{j+1}},$
$$\nu_j(\rho_j) := -1 + q_{j-1}q_j\sqrt{q_{j+1}} - 2q_{j-1}q_j^2q_{j+1} + 
q_{j-1}q_j^2q_{j+1}\sqrt{q_{j+1}}$$
$$ + \frac{q_{j-1}q_j^2\sqrt{q_{j+1}}}{q_{j+2}} - 
\frac{q_{j-1}q_j^2}{q_{j+2}^{2}q_{j+3}} \geq 0.
$$
By our assumptions on the sequence $(q_k)_{k=2}^\infty, $ and, since $j$ is an even number,
we have $q_j = q_{j+2} =a$ and $q_{j-1}= q_{j+1} =q_{j+3}= b,$
so the last inequality has the form
$$\nu_j(\rho_j) := -1 + ab\sqrt{b} - 2a^2b^2 + 
a^2b^2\sqrt{b} + ab\sqrt{b} - 
1 =$$
$$-2 + 2 ab\sqrt{b} - 2a^2b^2 + a^2b^2\sqrt{b} \geq 0.
$$

So we need the inequality
\[
\nu_j(\rho_j) = a^2 b^2 (\sqrt{b} - 2) + 2(ab \sqrt{b} - 1) \geq 0.
\]
Firstly, we consider the case when $b \geq 4.$ 
We have 
\[
a^2 b^2 (\sqrt{b} - 2) \geq 0,
\]
and
\[
(a b \sqrt{b} - 1) > 0.
\] 

Consequently, in the case
$ b \geq 4,$ the desired inequality $\nu_k(\rho_k) \geq 0$ is
proved.

Further, we consider the  case $3 \leq a < b <4.$

Under our assumptions we have $2 < \frac{2}{9} ab.$
Therefore,

\begin{align*}
\nu_j(\rho_j) & \geq   a^2 b^2 \sqrt{b} - 2a^2 b^2 + 2ab \sqrt{b} - \frac{2}{9} ab.
\end{align*}
So, we want to check that 
\[
\psi(a, b) := ab\sqrt{b} - 2a b + 2 \sqrt{b} - \frac{2}{9} \geq 0.
\]
Since $3 \leq a < b <4,$ we get
\[
\psi(a, b) = ab(\sqrt{b} - 2) + 2 \sqrt{b} - \frac{2}{9} \geq b^2(\sqrt{b} - 2) + 
2\sqrt{b} - \frac{2}{9}.
\]
Set $\sqrt{b} = t, t \geq 0,$ then we obtain the following inequality:
\[
t^5 - 2t^4 + 2t - \frac{2}{9} \geq 0.
\]
This inequality holds for $t \geq 0{.}11126$  (we used numerical methods 
 to find that the greatest real root of the polynomial on the left-hand 
side is less than $ 0{.}11126$). Thus, it follows that it holds
for $b \geq 0{.}01238$. 

Lemma~\ref{th:lm3} is proved. 
\end{proof}

To prove Theorem \ref{th:mthm1} we need one more lemma.

\begin{Lemma}
\label{th:lm3a} Let $\varphi(x) = \sum_{k=0}^\infty (-1)^k a_k x^k,$
$a_k >0,  k=0, 1, 2, \ldots ,$ be an entire function 
such that $$q_2 = q_4 = q_6 = \ldots = a \geq 3,$$
$$q_3 = q_5 = q_7 = \ldots = b > a.$$ Suppose that 
there exists $z_0 \in [0, \frac{a_1}{a_2}] = [0, a]$ such that $\varphi(z_0) \leq 0.$
For an arbitrary integer $j \geq 2$ we define 
$$r_j := q_2 q_3 \cdots  q_j z_0. $$ 
Then for every odd  $j =2m+1, m\in \mathbb{N},$  
the following inequality holds: 
$$ \varphi(r_j) \leq 0.$$
\end{Lemma}

\begin{proof}
Let us fix an arbitrary $j=2m+1, m\in \mathbb{N}.$ 
For $x \in [0,1]$ we have 
$$ 1 \geq x  > \frac{x^2}{q_2} >  \frac{x^3}{q_2^2 q_3}  > 
\frac{x^4}{q_2^3 q_3^2 q_4 } > \cdots  , $$
whence 
$$
 \varphi(x) >0 \quad \mbox{for all}\quad x\in [0,1].
$$

Thus, we have   $z_0 \in (1, a],$ whence $q_2 q_3 \cdots  q_j < r_j 
\leq q_2 q_3 \cdots  q_j q_{j+1.}$ 
For all  $ x  \in (q_2q_3 \cdots q_j, q_2q_3 \cdots q_j q_{j+1}]$ we have 
\begin{align*}
 1< x < \frac{x^2}{q_2} < \cdots < \frac{x^j}{q_2^{j-1}q_3^{j-2} \cdots q_j},  
 \end{align*}
and
\begin{align*}
 \frac{x^j}{q_2^{j-1}q_3^{j-2} \cdots q_j} \geq 
 \frac{x^{j+1}}{q_2^{j}q_3^{j-1} \cdots q_j^2 q_{j+1}} 
> \frac{x^{j+2}}{q_2^{j+1}q_3^{j} 
\cdots q_j^3 q_{j+1}^2 q_{j+2}} > \cdots.
\end{align*}
Thus, for every $ x  \in (q_2q_3 \cdots q_j, q_2q_3 \cdots q_j q_{j+1}]$ we have 
$$ \varphi(x) = 1 - x + \frac{x^2}{q_2} -  \ldots - \frac{x^{j-2}}{q_2^{j-3}q_3^{j-4} \cdots q_{j-2}}
+ \frac{x^{j-1}}{q_2^{j-2}q_3^{j-3} \cdots q_{j-2}^2 q_{j-1}} - $$
$$\frac{x^{j}}{q_2^{j-1}q_3^{j-2} \cdots q_{j-2}^3 q_{j-1}^2 q_j} + 
\frac{x^{j+1}}{q_2^{j}q_3^{j-1} \cdots q_{j-2}^4 q_{j-1}^3 q_j^2 q_{j+1}} 
- \ldots  \leq$$
$$  \frac{x^{j-1}}{q_2^{j-2}q_3^{j-3} \cdots q_{j-2}^2 q_{j-1}} - 
\frac{x^{j}}{q_2^{j-1}q_3^{j-2} \cdots q_{j-2}^3 q_{j-1}^2 q_j} + $$
$$\frac{x^{j+1}}{q_2^{j}q_3^{j-1} \cdots q_{j-2}^4 q_{j-1}^3 q_j^2 q_{j+1}} 
- \ldots  = \frac{x^{j-1}}{q_2^{j-2}q_3^{j-3} \cdots q_{j-2}^2 q_{j-1}} \cdot (1 -$$
$$ \frac{x}{q_2 q_3 \cdot q_{j-1}q_j} + \frac{x^2}{q_2^2 q_3^2 \cdot q_{j-1}^2q_j^2 q_{j+1}} 
-  \frac{x^3}{q_2^3 q_3^3 \cdot q_{j-1}^3q_j^3 q_{j+1}^2 q_{j+2}}  + \ldots).   $$
Pasting $x = r_j = q_2 q_3 \cdots  q_j z_0$ we obtain the inequality
$$   \varphi(r_j)  \leq \frac{r_j^{j-1}}{q_2^{j-2}q_3^{j-3} \cdots q_{j-2}^2 q_{j-1}} \cdot (1 -
z_0 +\frac{z_0^2}{q_{j+1}}- \frac{z_0^3}{q_{j+1}^2 q_{j+2}} +\ldots ) = $$
$$  \frac{r_j^{j-1}}{q_2^{j-2}q_3^{j-3} \cdots q_{j-2}^2 q_{j-1}} \cdot (1 -
z_0 +\frac{z_0^2}{a}- \frac{z_0^3}{a^2 b} +\ldots ) =$$
$$\frac{r_j^{j-1}}{q_2^{j-2}q_3^{j-3} \cdots q_{j-2}^2 q_{j-1}} \cdot \varphi(z_0) \leq 0.  $$

Lemma~\ref{th:lm3a} is proved. 
\end{proof}

Now, we are in a position to prove Theorem \ref{th:mthm1}. Suppose that
$\varphi(x) = \sum_{k=0}^\infty (-1)^k a_k x^k,$
$a_k >0,  k=0, 1, 2, \ldots ,$  $$q_2 = q_4 = q_6 = \ldots = a \geq 3,$$
$$q_3 = q_5 = q_7 = \ldots = b > a,$$ and
there exists $z_0 \in (1, \frac{a_1}{a_2}] = (1, a]$ such 
that $\varphi(z_0) \leq 0.$ 

Let us fix an arbitrary even $j = 2s, s\in \mathbb{N}.$  By Lemma~\ref{th:lm1}, 
the function $\varphi$ has exactly $j$ zeros in the circle $\{x\  :\   |x| <  \rho_j \},$
$\rho_j := q_2 q_3 \cdots  q_j \sqrt{q_{j+1}}. $ Our aim is to show that all
this zeros are real. We observe that $z_0 \in  [1, q_2],$ $\rho_2 =q_2 \sqrt{q_{3}} > z_0,$
$r_3 = q_2 q_3  z_0 > \rho_2, $ $\rho_4 = q_2 q_3  q_4 \sqrt{q_{5}} > r_3,$
$r_5 = q_2 q_3 q_4  q_5 z_0 > \rho_4, $  $\ldots ,$ $r_{j-1} = 
q_2 q_3 \cdots  q_{j-1} z_0 > \rho_{j-2},$  $\rho_j := q_2 q_3 \cdots  
q_j \sqrt{q_{j+1}} > r_{j-1} $
(see Lemma~\ref{th:lm3} and Lemma~\ref{th:lm3a} for the definitions 
of $\rho_j$ and $r_j$). Finally, we get
$$ 0< z_0 < \rho_2 < r_3 < \rho_4 < r_5 < \ldots <   r_{j-1} < \rho_j,   $$
and, by Lemma~\ref{th:lm3} and Lemma~\ref{th:lm3a} we have
$$\varphi(0) >0,\    \varphi (z_0) \leq 0,\    \varphi(\rho_2) \geq 0,\    
\varphi(r_3) \leq 0,  \     \varphi(\rho_4) \geq 0, $$
$$\varphi(r_5) \leq 0,  \        \ldots ,\      \varphi (\rho_{j-2}) \geq 0, \     
\varphi(r_{j-1}) \leq 0.$$
Thus, we find at least $j-1$ real zeros of $\varphi$ in the circle $\{x\  :\   |x| <  \rho_j \},$
and, since $\varphi$ has exactly $j$ zeros in this circle, all the zeros in the circle
are real.  By the fact that $\rho_j \to \infty$ for $j\to\infty,$
we conclude that all the zeros of $\varphi$ are real.

Theorem \ref{th:mthm1} is proved.

\section{Proof of Theorem \ref{th:mthm2}}
\begin{proof}

During the proof we will consider  the function  $\varphi_{a,b} (z) =f_{a,b}(-z)$ 
instead of $f_{a,b}.$

If $\varphi_{a,b}$ belongs to the Laguerre--P\'olya class, then there exists 
$z_0 \in  [0, \frac{a_1}{a_2}] = [0, a],$ such that $\varphi_{a,b}(z_0) \leq 0.$

As we mentioned in the proof of Lemma~\ref{th:lm3a}, for $x \in [0,1]$ we have 
$$ 1 \geq x  > \frac{x^2}{q_2} >  \frac{x^3}{q_2^2 q_3}  > 
\frac{x^4}{q_2^3 q_3^2 q_4 } > \cdots  , $$
whence 
\begin{equation}
\label{mthm1.1}
 \varphi_{a,b}(x) >0 \quad \mbox{for all}\quad x\in [0,1].
\end{equation}

Thus, we have   $z_0 \in (1, a].$
For any $x \in (1, a]$ we have

\[
\varphi_{a,b}(x) = 1 - x + \frac{x^2}{a} + \left(-\frac{x^3}{a^2b} + 
\frac{x^4}{a^3b^2a}\right) + \left(-\frac{x^5}{a^4b^3a^2b} + 
\frac{x^6}{a^5b^4a^3b^2a}\right) + \cdots
\]
We need the inequality:
\[
\frac{\partial}{\partial b} \left(-\frac{x^3}{a^2b} + 
\frac{x^4}{a^3b^2a}\right) = \frac{x^3}{a^2b^2} - \frac{2x^4}{a^4b^3} \geq 0,
\]
 which is equivalent to  $ x \leq \frac{a^2b}{2},$
that is correct for $x \in (1, a].$ Therefore, 
$
\left(-\frac{x^3}{a^2b} + \frac{x^4}{a^3b^2a}\right) 
$
is increasing in $b,$ which follows that 
\[
-\frac{x^3}{a^2b} + \frac{x^4}{a^3b^2a} \geq -\frac{x^3}{a^3} + \frac{x^4}{a^6},
\]
 
since, under our assumptions, $a < b.$

Analogously, for all $k\geq 2$

\begin{align*}
& \frac{\partial}{\partial b} \left(-\frac{x^{2k - 1}}{a^{2k - 2}b^{2k - 3} \cdots b} + 
\frac{x^{2k}}{a^{2k - 1}b^{2k - 2} \cdots a}\right) \\ 
&= \frac{\partial}{\partial b} \left(- \frac{x^{2k - 1}}{a^{(k-1)k} b^{(k-1)^2}} + 
\frac{x^{2k}}{a^{k^2} b^{(k-1)k}}\right) = \frac{(k-1)^2 \cdot 
x^{2k-1}}{a^{(k-1)k} \cdot b^{(k-1)^2 + 1}}\\
& \quad - \frac{(k-1)k \cdot x^{2k}}{a^{k^2} \cdot b^{(k-1)k +1}} \geq 0,
\end{align*}
or, equivalently, 
\[
x \leq \frac{k-1}{k} a^k b^{k-1},
\]
which is correct for $x \in (1, a]$ and $k\geq 2.$ Therefore, 
$$\left(- \frac{x^{2k - 1}}{a^{(k-1)k} b^{(k-1)^2}} + 
\frac{x^{2k}}{a^{k^2} b^{(k-1)k}}\right) \geq
\left(- \frac{x^{2k - 1}}{a^{\frac{(2k-1)(2k-2)}{2}} } + 
\frac{x^{2k}}{a^\frac{2k(2k-1)}{2} }\right) .$$

Thus, for any $x \in (1, a],$
\begin{align*}
\varphi_{a,b} (x) &= 1 - x + \frac{x^2}{a} - \frac{x^3}{a^2b} + 
\frac{x^4}{a^3b^2a} - \cdots\\
& \geq 1 - x + \frac{x^2}{a} - \frac{x^3}{a^3} + \frac{x^4}{a^6} - \cdots = 
\sum_{k=0}^\infty \frac{(-1)^k x^k}{a^\frac{k(k-1)}{2}}
= g_{\sqrt{a}}(-\sqrt{a}x),
\end{align*}
where $g_a$ is the partial theta-function.

Since $\varphi_{a,b}(z_0) \leq 0$ for $z_0 \in (1, a],$ we have $g_{\sqrt{a}}(-\sqrt{a}z_0) \leq 0.$ 
Denote by $y_0 := - \sqrt{a}z_0,$ then $y_0 \in (\sqrt{a}, (\sqrt{a})^3]. $ So, for the partial
theta-function $g_{\sqrt{a}}$ there exists a point  $y_0 \in (\sqrt{a}, (\sqrt{a})^3]$
such that $g_{\sqrt{a}}(y_0) \leq 0.$ By Theorem E, $g_{\sqrt{a}}\in \mathcal{L-P},$
whence $(\sqrt{a})^2 \geq q_{\infty}.$

Thus, we have proved statement 1 of Theorem \ref{th:mthm2}.

Statement 2 follows from the reasoning above. Suppose that the numbers 
$a, b, q_{\infty} \leq a <b,$ are such that $\varphi_{a,b} \in \mathcal{L-P}.$ 
By Theorem~\ref{th:mthm1} there exists a point $z_0 \in (1, a]$ such that
$\varphi_{a,b} (z_0) \leq 0.$ As we have proved below, for every fixed
$a$ and  $x \in (1, a]$ the expression for $\varphi_{a,b}$ is non-decreasing in $b.$
Thus, if  $ a < c< b$ then
$$ \varphi_{a,c}(z_0) \leq \varphi_{a,b}(z_0) \leq 0. $$ 
By Theorem \ref{th:mthm1}, $\varphi_{a,c} \in \mathcal{L-P},$ and 
we have proved the statement 2 of Theorem \ref{th:mthm2}.

Statement 3 can be proved analogously. For every $x \in (1, a]$ we have
$$
\varphi_{a,b}(x) = 1 - x + \left(\frac{x^2}{a}  -\frac{x^3}{a^2b}\right) + 
\left(\frac{x^4}{a^3b^2a}  - \frac{x^5}{a^4b^3a^2b} \right) $$
$$+ \left(\frac{x^6}{a^5b^4a^3b^2a} - \frac{x^7}{a^6b^5a^4b^3a^2b} \right) + \cdots
$$
We need the inequality
$$\frac{\partial}{\partial a} \left(\frac{x^2}{a} - 
\frac{x^3}{a^2b}\right) = - \frac{x^2}{a^2} + \frac{2x^3}{a^3b} \leq 0,$$
 which is equivalent to $ x \leq \frac{ab}{2},$
that is correct for $x \in (1, a].$ Therefore, 
$
\left(\frac{x^2}{a} - 
\frac{x^3}{a^2b}\right)
$
is decreasing in $a.$

Analogously, for all $k\geq 2$
\begin{align*}
& \frac{\partial}{\partial a} \left(\frac{x^{2k}}{a^{2k - 1}b^{2k - 2} \cdots b^2a} - 
\frac{x^{2k+1}}{a^{2k}b^{2k - 1} \cdots b^3 a^2 b}\right) \\ 
&= \frac{\partial}{\partial a} \left( \frac{x^{2k }}{a^{k^2} b^{k(k-1)}} - 
\frac{x^{2k+1}}{a^{k(k+1)} b^{k^2}}\right) = - \frac{k^2 \cdot 
x^{2k}}{a^{k^2+1} \cdot b^{k(k-1)}}\\
& \quad + \frac{k(k+1) \cdot x^{2k+1}}{a^{k^2 +k+1} \cdot b^{k^2}} \leq 0,
\end{align*}
or, equivalently, 
$$
x \leq \frac{k}{k+1} a^k b^{k},
$$
which is correct for $x \in (1, a]$ and $k\geq 2.$ Therefore, for every fixed
$b, b >a,$ and $x \in (1, a]$ the expression for $\varphi_{a,b}$ is non-increasing in $b.$
Suppose that the numbers 
$a, b, q_{\infty} \leq a <b,$ are such that $\varphi_{a,b} \in \mathcal{L-P}.$ 
By Theorem~\ref{th:mthm1} there exists a point $z_0 \in (1, a]$ such that
$\varphi_{a,b} (z_0) \leq 0.$ Thus, if  $ a < d < b$ then
$$ \varphi_{d,b}(z_0) \leq \varphi_{a,b}(z_0) \leq 0 $$ 
and $z_0\in (1, a] \subset (1, d].$
By Theorem \ref{th:mthm1}, $\varphi_{d,b} \in \mathcal{L-P},$ and 
we have proved the statement 3 of Theorem \ref{th:mthm2}.

Theorem \ref{th:mthm2} is proved.

\end{proof}

{\bf Aknowledgement.} \thanks{ The research was supported  by  the National 
Research Foundation of Ukraine funded by Ukrainian State budget in frames of 
project 2020.02/0096 ``Operators in infinite-dimensional spaces:  the interplay between 
geometry, algebra and topology''.}

\end{document}